\documentclass[11pt]{article}
\usepackage[utf8]{inputenc}
\usepackage{fullpage}
\usepackage{latexsym}
\usepackage{amsmath}
\usepackage{amssymb}
\usepackage{amsthm}
\usepackage{hyperref}
\usepackage{cite}
\usepackage{graphicx}
\usepackage{color}
\usepackage{framed}
\usepackage{xspace}
\usepackage{braket}
\usepackage{tikz}
\usepackage{physics}
\usepackage{dsfont}
\usepackage{mathrsfs}
\usepackage{subcaption}

\hypersetup{colorlinks=true,citecolor=blue,linkcolor=blue,filecolor=blue,urlcolor=blue,breaklinks=true}

\newtheorem{theorem}{Theorem}

\newtheorem{theorem*}{Theorem}

\newtheorem{observation}[theorem]{Observation}
\newtheorem{proposition}[theorem]{Proposition}
\newtheorem{definition}{Definition}

\newtheorem{fact}[theorem]{Fact}

\theoremstyle{definition}

\newtheorem{remark}[theorem]{Remark}

\newcommand{\GL}{\mathrm{GL}}

\newcommand{\F}{\mathbb{F}}
\newcommand{\Z}{\mathbb{Z}}
\newcommand{\ZC}{\mathrm{Z}}
\newcommand{\Q}{\mathbb{Q}}
\newcommand{\R}{\mathbb{R}}
\newcommand{\C}{\mathbb{C}}
\newcommand{\N}{\mathbb{N}}
\newcommand{\M}{\mathrm{M}}

\newcommand{\cA}{\mathcal{A}}
\newcommand{\cB}{\mathcal{B}}
\newcommand{\cC}{\mathcal{C}}

\newcommand{\bA}{\mathbf{A}}
\newcommand{\bB}{\mathbf{B}}

\newcommand{\fB}{\mathfrak{B}}

\title{Group-theoretic generalisations of vertex and edge connectivities
}

\author{Yinan Li \thanks{Centrum Wiskunde \& Informatica and QuSoft, Science Park 123, 1098XG Amsterdam, Netherlands ({\tt Yinan.Li@cwi.nl}). Partially supported by ERC Consolidator Grant 615307-QPROGRESS.}
\and 
Youming Qiao 
\thanks{Center for Quantum Software and Information, University of Technology Sydney, Ultimo NSW 2007, Australia. {\tt Youming.Qiao@uts.edu.au}. Partially supported by the Australian Research Council DECRA DE150100720. }
}
\begin{document}

\maketitle
\begin{abstract}
Let $p$ be an odd prime. Let $P$ be a finite $p$-group of class $2$ and exponent 
$p$, whose commutator quotient $P/[P,P]$ is of order $p^n$. We define two 
parameters for $P$ related to central decompositions. The first parameter, 
$\kappa(P)$, is the smallest integer $s$ for the existence of a subgroup $S$ of 
$P$ satisfying (1) $S\cap [P,P]=[S,S]$, (2) $|S/[S,S]|=p^{n-s}$, and (3) $S$ 
admits a non-trivial central decomposition. The second parameter, $\lambda(P)$, is 
the smallest integer $s$ for the existence of a central subgroup $N$ of order 
$p^s$, such that $P/N$ admits a non-trivial central decomposition.

While defined in purely group-theoretic terms, these two parameters generalise 
respectively the vertex and edge connectivities of graphs: For a simple undirected 
graph $G$, through the classical procedures of Baer (Trans. Am. Math. Soc., 1938), 
Tutte (J. Lond. Math. Soc., 1947) and Lov\'asz (B. Braz. Math. Soc., 1989), there 
is a $p$-group $P_G$ of class $2$ and exponent $p$ that is naturally associated 
with $G$. Our main results show that the vertex connectivity $\kappa(G)$ is equal 
to $\kappa(P_G)$, and the edge connectivity $\lambda(G)$ is equal to 
$\lambda(P_G)$. We also discuss the relation between $\kappa(P)$ and $\lambda(P)$ 
for a general $p$-group $P$ of class $2$ and exponent $p$, as well as the 
computational aspects of these parameters. 

\vskip 0.5em


\vskip 0.5em

\end{abstract}

\section{Introduction}

The main purpose of this note is to define and explore two natural group-theoretic 
parameters, 
which are closely related to vertex and edge connectivities in graphs.

In this 
introduction, we first introduce the
classical procedures of Baer~\cite{Bae38}, Tutte~\cite{Tut47}, and Lov\'asz~\cite{Lov89} 
which relate graphs with $p$-groups of class $2$ and exponent $p$.
We then define two group-theoretic parameters. Our
main 
result shows that the vertex and edge connectivities of a graph are equal to 
the 
two 
parameters we defined on the corresponding group respectively. We then compare 
the two parameters and discuss on their computational aspects. 

Since the main goal of this note is to set up a link between graph theory and 
group theory, we shall include certain background information, despite that it is 
well-known to researchers in the respective areas. 

\subsection{From graphs to groups: the Baer-Lov\'asz-Tutte 
procedure}\label{subsec:blt}

The route from graphs to groups, following 
Baer~\cite{Bae38}, Tutte~\cite{Tut47}, and Lov\'asz~\cite{Lov89}, goes via linear 
spaces of alternating matrices and alternating bilinear maps.

We set up some notation. For $n\in \N$, let $[n]:=\{1, \dots, n\}$. 
Let
$\binom{[n]}{2}$ be the set of size-$2$ subsets of $[n]$.
We use $\F$ to denote a field, and $\F_q$ 
to denote the finite field with $q$ elements. Vectors in $\F^n$ 
are column vectors, and $\langle\cdot\rangle$ denotes the linear 
span 
over underlying field $\F$.
Let $\Lambda(n, \F)$ be 
the linear space of $n\times n$ 
alternating matrices over $\F$. 
Recall that an $n\times n$ matrix $A$ over $\F$ is 
\emph{alternating} if for any $v\in \F^n$, $v^tAv=0$. That is, $A$ represents an 
alternating bilinear form. Subspaces $\cA$ of 
$\Lambda(n, \F)$, denoted by $\cA\leq\Lambda(n,\F)$,
are called alternating matrix spaces. 
Fix a field $\F$. For $\{i,j\}\in\binom{[n]}{2}$ with $i<j$,
the \emph{elementary} 
alternating matrix $A_{i,j}$ over $\F$ is the matrix with the $(i,j)$th entry 
being $1$, the 
$(j, i)$th entry being $-1$, and the rest entries being $0$. 

In this note, we only consider non-empty, simple, and undirected graphs with the 
vertex 
set being 
$[n]$. That is, a graph is $G=([n], E)$ where $E\subseteq 
\binom{[n]}{2}$. Let $|E|=m$. Note that the non-empty condition implies that 
$n\geq 2$ and $m\geq 1$. 

Let $p$ be an odd prime. We use $\fB_{p,2}$ to denote the class of 
\emph{non-abelian} $p$-groups of 
class $2$ and exponent $p$. That is, a non-abelian group $P$ is in $\fB_{p,2}$, if 
for 
any $g\in P$, $g^p=1$, and the commutator subgroup $[P,P]$ is contained in the 
centre $\ZC(P)$.
For $n, m\in \N$, we further define $\fB_{p,2}(n, m)\subseteq\fB_{p,2}$, which consists of those $P\in 
\fB_{p,2}$ with $|P/[P,P]|=p^n$ and 
$|[P,P]|=p^m$. Note that the non-abelian condition implies that $n\geq 2$ and 
$m\geq 1$ are required for $\fB_{p,2}(n, m)$ to be non-empty.

We then explain the procedure from graphs to groups in $\fB_{p,2}$ 
following Baer, Tutte and Lov\'asz.
\begin{enumerate}
\item\label{item: BLT step 1} Let $G=([n], E)$ be a simple and undirected graph with $m$ edges.
Following Tutte \cite{Tut47} and Lov\'asz \cite{Lov89}, we construct from $G$ an 
$m$-dimensional alternating matrix space in $\Lambda(n, \F)$ by setting 
\begin{equation}\label{eq:gf_to_sp}
\cA_G=\langle 
A_{i,j} : \{i,j\}\in E\rangle.
\end{equation}

\item\label{item: BLT step 2} Given an $m$-dimensional $\cA\leq\Lambda(n, \F)$,
let 
$\bA=(A_1, \dots, A_m)\in \Lambda(n, \F)^m$ be an ordered basis of $\cA$. The 
alternating bilinear map defined by $\bA$, $\phi_\bA:\F^n\times\F^n\to\F^m$, is 
\begin{equation}\label{eq:bil_def}
\phi_\bA(v, u)=(v^tA_1u, \dots, v^tA_mu)^t.
\end{equation}
Since $\cA$ is of dimension $m$, we 
have that $\phi_\bA(\F^n, \F^n)=\F^m$.

\item\label{item: BLT step 3} Let $p$ be an odd prime. 
Let 
$\phi:\F_p^n\times \F_p^n\to\F_p^m$ be an alternating bilinear map, such that 
$\phi(\F_p^n, \F_p^n)=\F_p^m$. 
Following Baer 
\cite{Bae38}, we 
define a $p$-group, $P_\phi\in \fB_{p,2}(n, m)$, as follows. The group 
elements are from $\F_p^n\oplus \F_p^m$. For $(v_i, u_i)\in \F_p^n\oplus \F_p^m$, 
$i=1, 2$, the group product $\circ$ is defined 
as
\begin{equation}\label{eq:gp_def}
(v_1, u_1)\circ (v_2, u_2) := (v_1+v_2, u_1+u_2+\frac{1}{2}\cdot \phi(v_1, v_2)).
\end{equation}
It can be verified that $P_\phi\in \fB_{p,2}(n, m)$, because 
of the condition that $\phi(\F_p^n, \F_p^n)=\F_p^m$. 
\end{enumerate}

Starting from a graph $G$, we follow the above three steps to obtain a 
$p$-group of class $2$ and exponent $p$, denoted by $P_G$. 
It can be verified 
easily that this process preserves isomorphism types, despite that the procedure 
from alternating matrix 
sapces to alternating bilinear maps depends on choices of ordered bases; see 
Remark~\ref{rem:trans}. That is, 
if the graphs $G_1$ 
and $G_2$ are isomorphic, then the corresponding $p$-groups $P_{G_1}$ and 
$P_{G_2}$ are isomorphic as well. 

\begin{definition}[The Baer-Lov\'asz-Tutte procedure]\label{def:blt}
Let $G=([n], E)$ be an undirected simple graph with $|E|=m>0$.
The 
Baer-Lov\'asz-Tutte procedure, as specified in the above three steps, takes $G$ 
and a prime $p>2$, 
and produces a $p$-group 
of class $2$ and exponent $p$, $P_G\in \fB_{p,2}(n, m)$.
\end{definition}

\subsection{Our results}\label{subsec:results}

\paragraph{Two group-theoretic parameters.} Let $H$ be a finite group. We use 
$J\leq H$ 
 to denote that $J$ is a 
 subgroup of $H$, and $J<H$ to denote that $J$ is a proper subgroup of $H$. 
For $S, T\subseteq H$, 
$ST=\{st:s\in S, t\in T\}$. If two subgroups $J, K\leq H$ 
satisfy that $JK=KJ$, then $JK$ is a subgroup of $H$. 

Recall that $H$ is a \emph{central product} of two subgroups $J$ and $K$, if (1) every 
element of $J$ commutes with every element of $K$, i.e. $[J,K]=1$, and (2)  
$H$ is generated by $J$ 
and 
$K$, i.e. $H=JK$. See e.g. 
\cite[pp. 137]{suzuki}.
In the following, 
we always assume that a central product 
is non-trivial, i.e., $J$ and $K$ are non-trivial proper subgroups of $H$. If such 
$J$ and $K$ exist, then we say that $H$ admits a central decomposition.

Given $P\in 
\fB_{p,2}$, a subgroup $S\leq P$ is \emph{regular} with respect to commutation, or 
simply regular for short, if 
$[S,S]=S\cap [P,P]$. 

\begin{definition}[$\kappa$ and $\lambda$ for $p$-groups of class $2$ and exponent 
$p$]
Let $P\in \fB_{p,2}(n, m)$. 

The \emph{regular-subgroup central-decomposition number} of $P$, denoted by 
$\kappa(P)$, 
is 
the 
smallest 
$s\in \N$ for the existence of a regular subgroup $S$ 
 with $|S/[S,S]|=p^{n-s}$, such 
that $S$ 
admits 
a central decomposition.

The \emph{central-quotient central-decomposition number} of $P$, denoted as 
$\lambda(P)$, is the smallest $s\in \N$ for the existence of a central subgroup 
$N$ of order $p^s$, such that $P/N$ admits a central decomposition.
\end{definition}
An explanation for imposing the regularity 
condition in the definition of $\kappa(P)$ can be found in 
Remark~\ref{rem:regularity}.
In the definition of $\lambda(P)$, we can actually restrict
$N$ to be from those central subgroups contained in 
$[P,P]$ (cf. Observation~\ref{obs:def} (2)). 

\paragraph{The results.} Recall that for a graph $G$, the \emph{vertex connectivity} 
$\kappa(G)$ denotes the smallest number of vertices needed to remove to disconnect 
$G$, and the \emph{edge connectivity} $\lambda(G)$ denotes the smallest number of edges 
needed to remove to disconnect $G$ \cite{Die06}. 

Given the above preparation, we can state our main result. 
\begin{theorem}\label{thm:main}
For an $n$-vertex and $m$-edge graph $G$, let $P_G\in \fB_{p,2}(n, m)$ be 
the result of applying the Baer-Lov\'asz-Tutte procedure to $G$ and a prime $p>2$. 
Then $\kappa(G)=\kappa(P_G)$, and 
$\lambda(G)=\lambda(P_G)$. 
\end{theorem}
Recall that $\kappa(P_G)$ and $\lambda(P_G)$ are defined in purely group-theoretic 
terms, while $\kappa(G)$ and $\lambda(G)$ are classical notions in graph theory. 
Therefore, Theorem~\ref{thm:main} sets up a surprising link 
between group theory and graph theory. 

To understand these two parameters and their relation better, we consider the 
following question. Recall that for a graph $G$, it is well-known that 
$\kappa(G)\leq \lambda(G)\leq \delta(G)$, 
where $\delta(G)$ denotes the minimum degree of vertices in $G$ (cf. e.g. 
\cite[Proposition 1.4.2]{Die06}). We study a question of the same type in 
the context of $p$-groups of class $2$ and exponent $p$. For this we need the 
following definition.
\begin{definition}[Degrees and $\delta$ for $p$-groups of 
class $2$ and exponent $p$]\label{def:gp_degree}
For 
$P\in\fB_{p,2}(n, m)$ and $g\in P$, suppose $C_P(g)=\{h\in P : 
[h, g]=1\}$
is of order $p^d$. Then the \emph{degree} of $g$ is 
$\deg(g)=n+m-d$. 
The \emph{minimum degree} of $P$, $\delta(P)$, is the minimum degree over $g\in 
P\setminus [P,P]$.
\end{definition}
It is easy to see that for any $g\in P$, $\deg(g)\leq n-1$ 
(cf. Observation~\ref{obs:def} (3)). Therefore 
$\delta(P)\leq n-1$. We 
then have the following. 

\begin{proposition}\label{prop:relation}
\begin{enumerate}
\item For any $P\in \fB_{p,2}$, $\kappa(P)\leq 
\delta(P)$, and $\lambda(P)\leq \delta(P)$.
\item There exists $P\in\fB_{p,2}$, such that 
$\kappa(P)>\lambda(P)$. 
\end{enumerate}
\end{proposition}
That is, while we can still upper bound $\kappa(P)$ and $\delta(P)$ using a 
certain 
minimum degree notion, the inequality 
$\kappa(\cdot)\leq\lambda(\cdot)$ does not hold in general in the $p$-group 
setting. 

\subsection{Related works and open ends}\label{subsec:discussion}

\paragraph{Related works.} Alternating matrix 
spaces and alternating bilinear maps serve as the intermediate objects  between 
graphs and groups in the Baer-Lov\'asz-Tutte procedure. 
We elaborate more on the previous works 
that demonstrate their links to the 
two sides.


The link between graphs and alternating matrix spaces  dates back to the works of 
Tutte and Lov\'asz \cite{Tut47,Lov89} in the context of perfect matchings. 
Let $G=([n], E)$ be a graph, and let $\cA_G\leq \Lambda(n, \F)$ be the 
alternating matrix space 
associated with $G$ as in Step~\ref{item: BLT step 1}.
Tutte and Lov\'asz realised 
that the matching number of $G$, 
$\mu(G)$, is equal 
to the maximum rank over matrices in $\cA_G$.\footnote{This is straightforward to 
see if the underlying field $\F$ is large enough. If $\F$ is small, it follows 
e.g. as a consequence of the linear matroid parity theorem; cf. the discussion 
after \cite[Theorem 4]{Lov89}.}
More specifically, Tutte 
represented $G$ as a symbolic matrix, that is a matrix whose entries are either 
variables or $0$ \cite{Tut47}. It can be interpreted as a linear space of matrices 
in a straightforward fashion. Lov\'asz then more 
systematically 
studied this construction from the latter perspective \cite{Lov89}.

Recently in \cite{BCG+19}, the 
second author and collaborators showed 
that the independence number of $G$, 
$\alpha(G)$, is 
equal to the 
maximum 
dimension over the isotropic spaces\footnote{A subspace $U\leq 
\F^n$ is an 
isotropic space of $\cA\leq \Lambda(n, \F)$, if for any $u, u'\in U$, and any 
$A\in \cA$, $u^tAu'=0$.} of $\cA_G$. They also showed 
that the 
chromatic number of $G$, $\chi(G)$, is equal to the minimum $c$ such that there 
exists a direct sum decomposition of $\F^n$ into $c$ non-trivial isotropic spaces 
for $\cA_G$. 
As the reader will see below, the proof of Theorem~\ref{thm:main} 
also goes by defining appropriate parameters $\kappa(\cdot)$ and $\lambda(\cdot)$ 
for 
alternating matrix 
spaces, and proving that $\kappa(\cA_G)=\kappa(G)$ and 
$\lambda(\cA_G)=\lambda(G)$. This translates another two graph-theoretic 
parameters to the alternating matrix space setting.

The work most relevant to the current note in this direction is \cite{LiQ2017} by 
the 
present authors. In 
that work, we adapted a combinatorial technique for the graph isomorphism problem 
by Babai, 
Erd\H{o}s, and Selkow \cite{BES80}, to tackle isomorphism testing of groups from 
$\fB_{p,2}$, via alternating matrix spaces. This leads to the definition of a 
``cut'' for alternating matrix spaces, which in turn naturally leads to the edge 
connectivity notion; cf. the proof of 
Proposition~\ref{prop:lambda_G_cA}.

The link between alternating bilinear maps
and $\fB_{p,2}$ dates back to the work
of Baer \cite{Bae38}. That is, from an alternating bilinear map $\phi$, we can 
construct a group $P_\phi$ in $\fB_{p,2}$ as in Step~\ref{item: BLT step 3}. 
On the 
other hand, given  $P\in \fB_{p,2}(n, m)$, by taking
 the commutator bracket we obtain an alternating 
bilinear map $\phi_P$. 
A generalisation of this link to $p$-groups of Frattini 
class $2$ was crucial in Higman's enumeration of $p$-groups
\cite{Hig60}. 
Alperin \cite{Alp65}, Ol'shanskii 
\cite{Ols78} and Buhler, Gupta, and Harris \cite{BGH87} used this link to study 
large 
abelian subgroups of $p$-groups, a question first considered by Burnside 
\cite{Bur13}. This 
is because abelian subgroups of $P$ containing $[P,P]$ correspond to isotropic 
spaces of $\phi_P$. 

The works most relevant to the current note in this direction are 
\cite{Wil09,Wil09b} by James B. Wilson. He studied 
central decompositions of $P$ via the link between alternating 
bilinear maps and $\fB_{p,2}$. In particular, he utilised that 
central decompositions of $P$ correspond to orthogonal decompositions of $\phi_P$.

Finally, we recently learnt of the work \cite{Rossmann19} of Rossmann and Voll, 
who 
study those $p$-groups of class 
$2$ and exponent $p$ obtained from graphs through the Baer-Lov\'asz-Tutte
procedure in the context of zeta functions of groups.

\paragraph{Open ends.} The most interesting questions to us are the computational 
aspects of these parameters. That is, given the linear basis of an alternating 
matrix space $\cA\leq \Lambda(n, \F)$, compute $\kappa(\cA)$ and $\lambda(\cA)$ 
(see Definition~\ref{def:cA}). 
When $\F=\F_q$ with $q$ odd, there is a randomised 
polynomial-time algorithm to 
decide whether $\kappa(\cA)=\lambda(\cA)=0$ by Wilson \cite{Wil09b}. When $\F=\R$ 
or $\C$, by utilising certain machineries from \cite{IQ19}, Wilson's algorithm can 
be adapted to yield a deterministic polynomial-time algorithm to decide 
whether $\kappa(\cA)=\lambda(\cA)=0$. However, to directly use Wilson's algorithm 
to compute 
$\kappa(\cA)$ or $\lambda(\cA)$ seems difficult, as when 
$\kappa(\cA)=\lambda(\cA)=0$, a non-trivial orthogonal decomposition can be nicely 
translated to a certain idempotent in an involutive algebra associated with any 
linear basis of $\cA$; for details, see \cite{Wil09b}.

%

\section{Proofs}\label{sec:proof}

\subsection{Preparations}\label{subsec:prep}
%
%
%

Some notation has been introduced at the beginning 
of sections~\ref{subsec:blt} 
and~\ref{subsec:results}. We add some more here. For a field $\F$ and $d, e\in 
\N$, we use $\M(d\times e, \F)$ to denote the linear space of 
$d\times e$ matrices over $\F$, and $\M(d,\F):=\M(d\times d, \F)$. The $i$th 
standard basis vector of $\F^n$ is 
denoted by ${e_i}$.
\paragraph{Some notions for alternating matrix spaces.} We introduce some basic 
concepts, and then define $\kappa$ and 
$\lambda$, for alternating matrix spaces.

Let $\cA, \cB\leq \Lambda(n, \F)$. We say that $\cA$ and $\cB$ are 
\emph{isometric}, if 
there 
exists $T\in\GL(n, \F)$, such that $\cA=T^t\cB T:=\{T^tBT: B\in \cB\}$. 


For a $d$-dimensional $W\leq \F^n$, let $T$ be an $n\times d$ matrix whose columns 
span $W$. Then the restriction of $\cA$ to $W$ via $T$ is $\cA|_{W, T}:=\{T^tAT : 
A\in \cA\}\leq \Lambda(d, \F)$. For a different $n\times d$ matrix $T'$ whose 
columns also span $W$, $\cA|_{W, T'}$ is isometric to $\cA|_{W, T}$. So we can 
write $\cA|_W$ to indicate a restriction of $\cA$ to $W$ via some such $T$.

Let $\cA\leq \Lambda(n, \F)$ be of dimension $m$. 
An \emph{orthogonal 
decomposition} of $\cA$ is a 
direct sum 
decomposition of $\F^n$ into $U\oplus V$, such that for any $u\in U$, $v\in V$, 
and $A\in \cA$, $u^tAv=0$. An orthogonal decomposition is non-trivial, if neither 
$U$ nor $V$ is the 
trivial space. In the following, we always assume an orthogonal decomposition to 
be non-trivial unless otherwise stated. 

In the degenerate case when $\cA\leq \Lambda(n, \F)$ is the zero space, we define 
it to have an orthogonal decomposition for any $n\in \N$.
When $\cA=\langle A\rangle \leq \Lambda(n, \F)$ is of dimension $1$ and $n>2$, 
$\cA$ always admits an orthogonal decomposition. This can be seen easily from the 
canonical form
for alternating matrices \cite[Chap. XV, Sec. 8]{Lang}.

\begin{definition}[$\kappa$ and $\lambda$ for alternating matrix 
spaces]\label{def:cA}
Let $\cA\leq \Lambda(n, \F)$ be of dimension $m$. 

We define the 
\emph{restriction-orthogonal number} of $\cA$, $\kappa(\cA)$, as the 
minimum $c\in \N$ for the existence of a dimension-$(n-c)$ subspace $W\leq \F^n$, 
such that $\cA|_W$ admits an orthogonal decomposition. 

We define the 
\emph{subspace-orthogonal number} of $\cA$, $\lambda(\cA)$, as the 
minimum $c\in \N$ for the existence of a dimension-$(m-c)$ subspace $\cA'\leq 
\cA$, such that $\cA'$ admits an orthogonal decomposition. 
\end{definition}

Clearly, $\cA$ itself admits an orthogonal decomposition if and only if 
$\kappa(\cA)=\lambda(\cA)=0$. Since we defined the zero alternating matrix space 
to 
have an 
orthogonal decomposition, $\kappa(\cA)\leq n-1$ and $\lambda(\cA)\leq m$.

Suppose we are given a dimension-$m$ $\cA=\langle A_1, \dots, A_m\rangle\leq 
\Lambda(n, \F)$. We form a $3$-tensor $\mathtt{A}\in 
\F^{n\times n\times m}$ such that $\mathtt{A}(i,j,k)=A_k(i,j)$. We illustrate 
the 
existence of an orthogonal decomposition for $\cA$, the existence of $W$ such that 
$\cA|_W$ has an orthogonal decomposition, and the existence of $\cA'\leq \cA$ with 
an orthogonal decomposition, up to appropriate basis changes, in 
Figure~\ref{fig: 3-tensor formulation}.

\begin{figure}[ht]
    \centering
    \begin{subfigure}[b]{0.28\textwidth}
        \includegraphics[width=.9\textwidth]{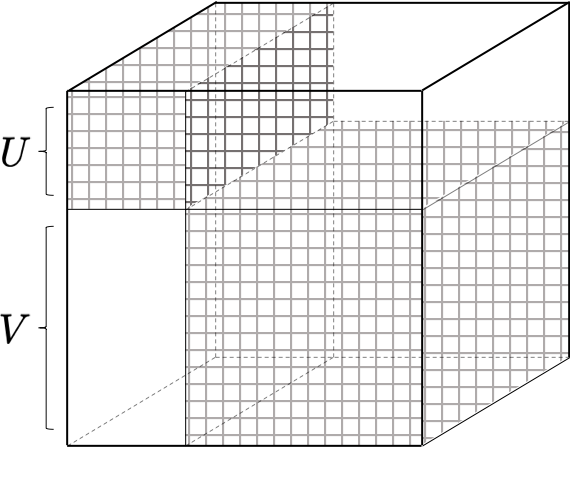}
        \caption{$\kappa(\cA)=\lambda(\cA)=0$}
        \label{fig: disconnected}
    \end{subfigure}
\qquad   ~ 
    \begin{subfigure}[b]{0.3\textwidth}
        \includegraphics[width=\textwidth]{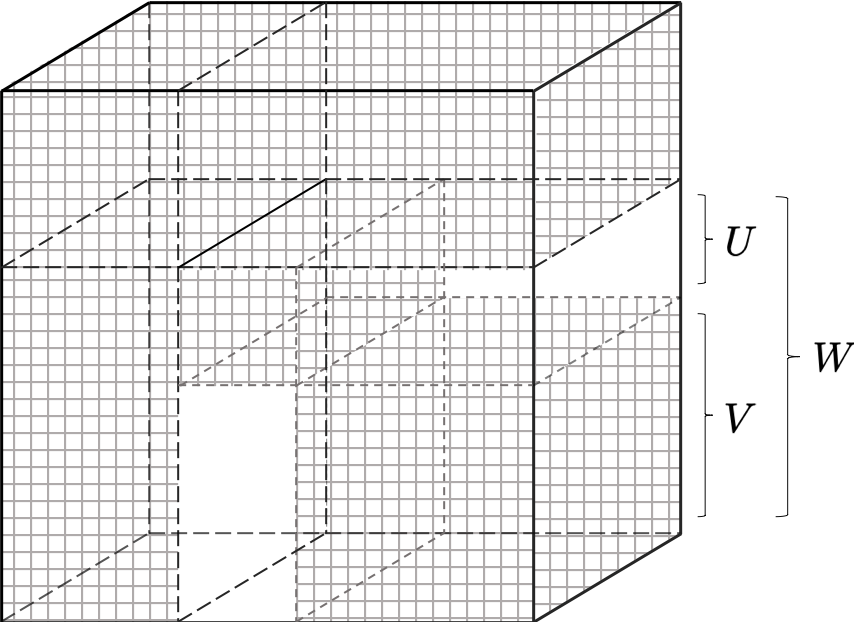}
        \caption{$\kappa(\cA)\leq n-\dim(W)$}
        \label{fig: vertex connectivity}
    \end{subfigure}
    ~ 
    \begin{subfigure}[b]{0.3\textwidth}
        \includegraphics[width=\textwidth]{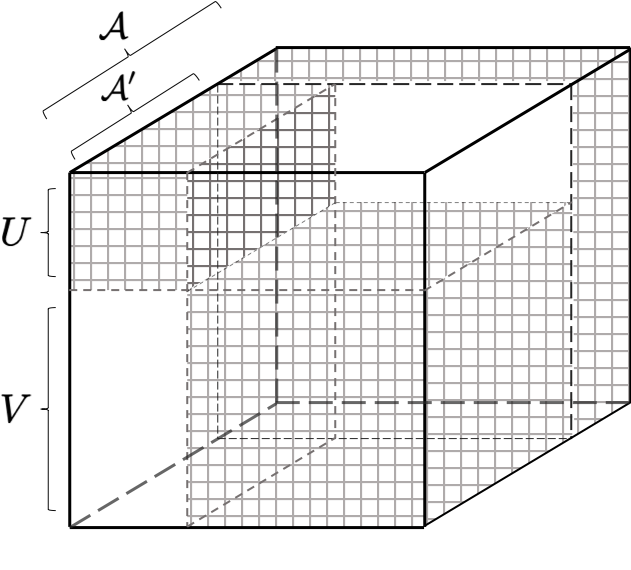}
        \caption{$\lambda(\cA)\leq m-\dim(\cA')$}
        \label{fig: edge connectivity}
    \end{subfigure}
    \caption{Pictorial descriptions of the alternating matrix space 
    parameters. The white regions indicate that the entries there are all zero. 
    For example, in (a), suppose $U\oplus V$ is an orthogonal decomposition for 
    $\cA$. Then up to a change of basis, the upper-right and the lower-left 
    corners of $\mathtt{A}$ have all-zero entries. (b) and (c) also indicate the 
    situations with appropriate changes of bases.}\label{fig: 
    3-tensor formulation}
\end{figure}

\paragraph{Some notions for alternating bilinear maps.} We introduce basic 
concepts, and then define $\kappa$ and 
$\lambda$, for alternating bilinear maps.

Let $\phi, 
\psi:\F^n\times \F^n\to \F^m$ be two alternating bilinear maps. Following 
\cite{Wil09}, we say that $\phi$ 
and 
$\psi$ are \emph{pseudo-isometric}, if they are the same under the natural action of 
$\GL(n, \F)\times\GL(m, \F)$. 

For $U\leq \F^n$, $\phi$ naturally restricts to $U$ to give $\phi|_U:U\times U\to 
\F^m$. For $X\leq \F^m$, $\phi$ naturally induces $\phi/_X:\F^n\times \F^n\to 
\F^m/X$ by 
composing 
$\phi$ with the projection from $\F^m$ to $\F^m/X$. 

Let $\phi:\F^n\times \F^n\to \F^m$ be an alternating bilinear map. An orthogonal 
decomposition of $\phi$ is a direct sum decomposition of $\F^n=U\oplus V$, 
such 
that for any $u\in U$, $v\in V$, we have $\phi(u, v)=0$. In the following, unless 
stated otherwise, we always 
assume an orthogonal decomposition of $\phi$ to be non-trivial, i.e., neither $U$ 
nor $V$ is the trivial space. 
\begin{definition}[$\kappa$ and $\lambda$ for alternating bilinear maps]
Let $\phi:\F^n\times \F^n\to \F^m$ be an alternating bilinear map. 

The \emph{restriction-orthogonal number} of $\phi$, $\kappa(\phi)$, is the minimum 
$c\in\N$ for the existence of a dimension-$(n-c)$ subspace $U\leq \F^n$, such that 
$\phi|_U$ admits an orthogonal decomposition. 

The \emph{quotient-orthogonal 
number} of $\phi$, $\lambda(\phi)$, is the minimum $c\in \N$ for the existence of 
a dimension-$c$ $X\leq  \F^m$, such that $\phi/_X$ admits an orthogonal 
decomposition.
\end{definition}

\begin{remark}[From alternating matrix spaces to bilinear maps]\label{rem:trans}
This connection is simple but may deserve some discussion. Recall that, given an 
$m$-dimensional alternating matrix space $\cA\leq \Lambda(n, 
\F)$, we can fix an ordered basis of $\cA$ as $\bA=(A_1, \dots, A_m)\in \Lambda(n, 
\F)^m$, and construct an alternating bilinear map $\phi_\bA:\F^n\times\F^n\to\F^m$ 
as in Equation~\ref{eq:bil_def}. Furthermore, $\phi_\bA(\F^n, \F^n)=\F^m$ because 
$\cA$ is of dimension $m$. 
%

In the above transformation, we shall need $\bA\in\Lambda(n, \F)^m$ as an 
intermediate object. For a different ordered bases $\bA'$, $\phi_{\bA'}$ is 
pseudo-isometric to $\phi_{\bA}$. Because of this, we shall write $\phi_\cA$ to 
indicate $\phi_\bA$ with some ordered basis $\bA$ of $\cA$.

%
Furthermore, if $\cA$ and $\cB$ are isometric and $\bA$ (resp. $\bB$) is an 
ordered basis for $\cA$ (resp. $\cB$), then $\phi_\cA$ and  $\phi_\cB$ are 
pseudo-isometric as well.
\end{remark}


%
%
%

\subsection{Proof of Theorem~\ref{thm:main}}

The proof of Theorem~\ref{thm:main} goes by showing that the 
parameters $\kappa(\cdot)$ and $\lambda(\cdot)$ defined for graphs, alternating 
matrix spaces, alternating bilinear maps, and groups from $\fB_{p,2}$, are 
preserved in the three steps of the Baer-Lov\'asz-Tutte procedure. The first step, 
from graphs to alternating matrix spaces, is the tricky one, at least for 
$\lambda(\cdot)$. The other two 
steps are rather straightforward.

\paragraph{From graphs to alternating matrix spaces.} 

\begin{proposition}\label{prop:kappa_G_cA}
Let $G=([n], E)$ be a graph, and let $\cA_G\leq \Lambda(n, \F)$ be defined in 
Step~\ref{item: BLT step 1}.
Then $\kappa(G)=\kappa(\cA_G)$.
\end{proposition}
\begin{proof}
We first show $\kappa(\cA_G)\leq \kappa(G)$. Let $I\subseteq[n]$
be a subset 
of 
vertices of size $d=n-\kappa(G)$, such that the induced subgraph of $G$ on $I$ is 
disconnected. Let $W=\langle {e_i}:i\in I\rangle$, and $T$ be 
the $n\times d$ matrix over $\F$ whose columns are ${e_i}\in \F^n$, $i\in I$.
It is straightforward to verify that $\cA_G|_{W,T}$ admits an orthogonal 
decomposition.

We then show $\kappa(\cA_G)\geq\kappa(G)$. Let $W\leq \F^n$ be a subspace of 
dimension $d=n-\kappa(\cA_G)$, such that $\cA|_W$ admits an orthogonal 
decomposition. That is, there exists $W=U\oplus V$ such that 
\begin{equation}\label{eq:ortho}
\forall u\in U, v\in 
V, \forall A\in \cA, u^tAv=0.
\end{equation}
 Suppose $\dim(U)=b$ and $\dim(V)=c$, so 
$d=b+c$. Construct an $n\times d$ matrix $T=\begin{bmatrix}T_1 &T_2\end{bmatrix}$ where $T_1$ (resp. $T_2$) 
is of size 
$n\times b$ (resp. $n\times c$) and its columns form a basis of $U$ (resp. $V$).
Let the $i$th row of $T_1$ be $u_i^t$ where $u_i\in \F^b$, and let 
the $j$th row of 
$T_2$ be $v_j^t$ where $v_j\in \F^c$, 
for $i,j\in[n]$. Then by Equation~\ref{eq:ortho}, for any 
$\{i,j\}\in E$,
\begin{equation}\label{eq:kappa_G}
T_1^t({e_i}e_j^t-{e_j}e_i^t)T_2=u_iv_j^t-u_jv_i^t
\end{equation}
is the all-zero matrix of 
size $b\times c$. 

Because $T$ is of rank $d$, there exists a $d\times d$ submatrix 
$R$ of 
$T$ of rank $d$. Let $I\subseteq[n]$ be the set of row indices 
of this submatrix 
$R$. 
We claim that the induced subgraph of $G$ on $I$, 
$G[I]$, is disconnected. 
To show this, we exhibit a partition of $I=I_1\uplus I_2$ such that no edges in 
$G[I]$ 
go 
across $I_1$ and $I_2$. 

Recall that $R$ is of rank $d$. As an easy consequence of the Laplace expansion, 
there exists a partition of $I$, $I=I_1\uplus I_2$ with $|I_1|=b$, $|I_2|=d-b=c$, such that 
the following holds. Let $R_1$ 
be the $b\times b$ submatrix of $R$ with row indices 
from $I_1$ and column indices from $[b]$, and $R_2$ 
the $c\times c$ submatrix of $R$ 
with row indices from $I_2$ and column indices from $[d]\setminus [b]$. Then $R_1$ 
and 
$R_2$
are both full-rank. 
Note that $\{u_i^t:i\in I_1\}$ is the set of rows of $R_1$ and 
$\{v_j^t:j\in I_2\}$ is the set of rows of $R_2$.

We then claim that no edges in $G[I]$ 
go across $I_1$ and $I_2$. 
By contradiction, suppose there is an edge $\{i, j\}$, $i\in I_1$ and $j\in I_2$, in 
$G[I]$. 
Then the same edge $\{i,j\}$ is also in $G$.
By Equation~\ref{eq:kappa_G}, we have 
$u_iv_j^t-u_jv_i^t$ is the all-zero matrix. Since $R_1$ and $R_2$ are full-rank, we 
have $u_i$ and $v_j$ are nonzero vectors. This implies that $u_j=\alpha u_i$ and 
$v_i=(1/\alpha) v_j$ for some nonzero $\alpha\in \F$. But this implies that 
$\begin{bmatrix}
u_j^t & v_j^t
\end{bmatrix}=\alpha \begin{bmatrix}
u_i^t & v_i^t
\end{bmatrix}$, that is,  the $i$th and $j$th rows of $T$
are linearly 
dependent. Noting that these rows are in $R$ which is full-rank, we arrive at the 
desired contradiction. This concludes the proof. 
\end{proof}

\begin{proposition}\label{prop:lambda_G_cA}
Let $G=([n], E)$ be a graph, and let $\cA_G\leq \Lambda(n, \F)$ be defined in 
Step~\ref{item: BLT step 1}. Then $\lambda(G)=\lambda(\cA_G)$.
\end{proposition}
\begin{proof} 
We first show $\lambda(\cA_G)\leq\lambda(G)$. Let $D$
be a size-$\lambda(G)$ 
subset of $E$ such that $G'=([n],E\setminus D)$ is disconnected. Let 
$\cA_{G'}=\langle A_{i,j}:\{i,j\}\in E\setminus D\rangle\leq \cA_G$.
It is 
straightforward to verify that $\cA_{G'}$ 
admits an orthogonal decomposition.

We then show $\lambda(\cA_G)\geq\lambda(G)$. 
For this, it is convenient to 
introduce an equivalent 
formulation of $\lambda(\cdot)$ for alternating matrix spaces, which is originated 
from \cite{LiQ2017}.

Given a direct sum 
decomposition 
$\F^n=U\oplus V$ with $\dim(U)=b$ and $\dim(V)=c=n-b$, let $T_1$ (resp.
$T_2$) be a $n\times b$ (resp. $n\times c$) matrix whose columns form a basis of $U$ (resp. $V$). Given an 
$m$-dimensional $\cA\leq \Lambda(n, \F)$, let 
$\cC_{U,V,T_1, T_2}(\cA)=\{T_1^tA T_2:A\in \cA\}\leq 
\M(b\times c, \F)$. 
Note that 
different choices of $T_1$ and $T_2$ result in a subspace of $\M(b\times c, \F)$ 
which can be transformed to $\cC_{U, V, T_1, T_2}(\cA)$ by left-multiplying some 
$X\in \GL(b, \F)$ and right-multiplying some $Y\in\GL(c, \F)$. So we can write $\cC_{U, 
V}$ to indicate $\cC_{U, V, T_1, T_2}$ via some such $T_1$ and $T_2$. 
We claim that 
\begin{equation}\label{eq:alt_def}
\lambda(\cA)=\min\{\dim(\cC_{U,V}(\cA)):\forall\text{ non-trivial } 
\F^n=U\oplus V\}.
\end{equation}
To see 
this, let $\cA'\leq\cA$ be of dimension $m-\lambda(\cA)$ which 
admits an orthogonal decomposition $\F^n=U\oplus V$. It is easy to verify that 
$\dim(\cC_{U,V}(\cA))\leq m-(m-\lambda(\cA))=\lambda(\cA)$. On the other hand, let 
$\F^n=U\oplus V$ be a direct sum decomposition such that $\dim(\cC_{U,V}(\cA))$ is 
minimal. 
Let $T_1$ (resp.
$T_2$) be a matrix whose columns form a basis of $U$ (resp. $V$).
Let 
$\cA'=\{A\in\cA:T_1^tAT_2=0\}$. 
We then have $\dim(\cA')=m-\dim(\cC_{U,V}(\cA))$, 
and clearly
$\cA'$ admits an orthogonal decomposition. This gives $\lambda(\cA)\leq 
m-\dim(\cA')=\dim(\cC_{U,V}(\cA))$.

After introducing this formulation, let $\F^n=U\oplus V$ be a direct sum 
decomposition with 
$\dim(U)=b$ and $\dim(V)=c=n-b$, such that 
$\dim(\cC_{U,V}(\cA_G))=\lambda(\cA_G)=d$. 
Construct an $n\times n$ full-rank matrix $T=\begin{bmatrix}
T_1 & T_2
\end{bmatrix}$ where $T_1$ (resp.
$T_2$) is a $n\times b$ (resp. $n\times c$) matrix whose columns form a basis of $U$ (resp. $V$).
Let the $i$th row of $T_1$ be $u_i^t$ where $u_i\in \F^b$, and let the $j$th row 
of $T_2$ be $v_j^t$ where $v_j\in \F^c$. 
We distinguish two cases:
\begin{enumerate}
\item Suppose for any $i\in[n]$, $u_i\neq 0$ if and only if $v_i=0$. Then there 
exists $[n]=I_1\uplus I_2$ with $|I_1|=b$ and $|I_2|=c$, such that $i\in I_1$ if 
and only if $u_i\neq 
0$, and $j\in I_2$ if and only if $v_j\neq 0$. Furthermore, vectors in $\{u_i: i\in I_1\}$ are 
linearly independent, and vectors in $\{v_j: j\in I_2\}$ are linearly independent.
We claim that there is no more than $d$ edges of $G$ crossing $I_1$ and $I_2$. Suppose 
not, then there exists 
$\{\{i_1, j_1\}, \dots, \{i_{d+1}, j_{d+1}\}\}\subseteq E$, such that $i_k\in I_1$, 
and $j_k\in I_2$ for $k\in[d+1]$. Note that 
\begin{equation}\label{eq:proof_lambda}
T_1^t({e_{i_k}}e_{j_k}^t-{e_{j_k}}e_{i_k}^t)T_2
=u_{i_k}v_{j_k}^t-u_{j_k}v_{i_k}^t
=u_{i_k}v_{j_k}^t\in \cC_{U,V}(\cA_G)
\end{equation}
for all $k\in[d+1]$. It is straightforward to verify that $u_{i_k}v_{j_k}^t$, 
$k\in[d+1]$, are linearly independent, which contradicts that $\cC_{U,V}(\cA_G)$ 
is of 
dimension $d$. 
\item Suppose there exists $i\in[n]$, such that both $u_{i}$ and $v_{i}$ are 
nonzero. 
Suppose by contradiction that $\lambda(G)>d$. It follows that the vertex $i$ is of 
degree at least $d+1$. 
Suppose $i$ is adjacent to $j_1, 
\dots, 
j_{d+1}\in [n]$. Then $u_{i}v_{j_k}^t-u_{j_k}v_{i}^t\in \cC_{U,V}(\cA_G)$ for 
$k\in[d+1]$ by Equation~\ref{eq:proof_lambda}. Since $\dim(\cC_{U,V}(\cA_G))=d$, 
the matrices
$u_{i}v_{j_k}^t-u_{j_k}v_{i}^t$, $k\in[d+1]$, are linearly dependent. It 
follows that there exist 
$\alpha_k\in\F$ for $k\in[d+1]$, at least one of which is nonzero, such that 
$$\sum_{k=1}^{d+1}\alpha_k(u_{i}v_{j_k}^t-u_{j_k}v_{i}^t)=0.$$ 
This implies that 
$$
u_{i}(\sum_{k=1}^{d+1}\alpha_kv_{j_k}^t)=(\sum_{k=1}^{d+1}\alpha_k 
u_{j_k})v_{i}^t$$
as two rank-$1$ matrices. 
From the above, and by the assumption that $u_{i}$ and $v_{i}$ are nonzero, we 
have 
that $\beta u_{i}= \sum_{k=1}^{d+1}\alpha_k u_{j_k}$ and 
$\beta v_{i}= \sum_{k=1}^{d+1}\alpha_kv_{j_k}$ for some nonzero $\beta\in \F$. 
Since at least one of $\alpha_k$'s is nonzero, this means that the rows in $T$ 
with indices 
$\{i, j_1, \dots, j_{d+1}\}$ are linearly dependent,
which contradicts that $T$ is full-rank. 
\end{enumerate}
These conclude the proof that $\lambda(\cA_G)\geq \lambda(G)$.
\end{proof}
\begin{remark}[Cuts in alternating matrix spaces]\label{rem:cut}
The alternative formulation of $\lambda(\cdot)$ as in Equation~\ref{eq:alt_def} 
rests on a natural generalisation of the notion of cuts in graphs. 
Proposition~\ref{prop:lambda_G_cA} 
then indicates that for an alternating 
matrix space $\cA_G$ constructed from a graph $G$, the minimum 
cut sizes of $\cA_G$ and $G$ are equal.
\end{remark}

\paragraph{From alternating matrix spaces to alternating bilinear maps.}
We now relate the parameters $\kappa(\cdot)$ and $\lambda(\cdot)$ for alternating 
matrix spaces and alternating bilinear maps in the following easy proposition. 
Note that we use the notation $\phi_\cA$ due to the discussions in 
Remark~\ref{rem:trans}.
\begin{proposition}\label{prop:space_map}
For an $m$-dimensional $\cA\leq \Lambda(n, \F)$, let an alternating bilinear map 
$\phi_\cA:\F^n\times \F^n\to\F^m$ be defined in Step~\ref{item: BLT step 2}.
 Then we have $\kappa(\cA)=\kappa(\phi_\cA)$, 
and $\lambda(\cA)=\lambda(\phi_\cA)$.
\end{proposition}
\begin{proof}
The equality $\kappa(\cA)=\kappa(\phi_\cA)$ is straightforward to verify. 

To show that $\lambda(\cA)\geq\lambda(\phi_\cA)$,
let 
$\cA'\leq \cA$ be a dimension-$(n-\lambda(\cA))$ subspace of $\cA$ admitting an 
orthogonal decomposition. Let $c=\lambda(\cA)$. We fix an ordered basis of $\cA$, 
$\bA=(A_1, \dots, A_m)$, 
such that $\{A_1, \dots, A_{m-c}\}$ spans $\cA'$. Let $X\leq \F^m$ be the linear 
span of the last $c$ standard basis vectors. 
We claim that $\phi_\bA/_X$ admits an 
orthogonal decomposition. 
Indeed, let $U\oplus V$ be an orthogonal decomposition of $\cA'$. Then for any 
$u\in U, v\in V$, we have $\phi_\bA(u, v)\in X$, which means that $U\oplus V$ is 
also an orthogonal decomposition for $\phi_\bA/_X$. 

To show that $\lambda(\cA)\leq\lambda(\phi_\cA)$,
let $\bA=(A_1, \dots, A_m)$ be 
an ordered basis of $\cA$, and let $c=\lambda(\phi_\cA)$. Let $X$ be a 
dimension-$c$ 
subspace of $\F^m$, such that $\phi_\bA/_X$ admits 
an orthogonal decomposition $\F^n=U\oplus V$. That is, for any $u\in U$ and $v\in 
V$, $\phi_\bA(u, v)\in X$. 
Form an ordered basis of 
$\F^m$, $(w_1, \dots, w_m)$, where $w_i=\begin{bmatrix}
w_{i,1} \\ \vdots \\ w_{i,m}
\end{bmatrix}\in \F^m$, 
such that the last $c$ vectors form a basis of $X$. Let 
$A_i'=\sum_{j\in[m]}w_{i,j}A_j$ be another ordered basis of $\cA$, and 
$\bA'=(A_1', \dots, A_m')$. Then for any 
$u\in U$ and $v\in V$, since $\phi_\bA(u, v)\in X$, the first $m-c$ entries of 
$\phi_{\bA'}(u, v)$ are zero. In 
particular, this implies that $\F^n=U\oplus V$ 
is an orthogonal decomposition for 
$\cA'=\langle A_1', \dots, A_{m-c}'\rangle$, where $\cA'$ is of dimension $m-c$.
%
\end{proof}

\paragraph{From alternating bilinear maps to groups from $\fB_{p,2}$.} 
%
To start with, we observe the following basic properties of $\kappa$, $\lambda$, 
and 
$\delta$ for groups from $\fB_{p,2}(n, m)$.
\begin{observation}\label{obs:def}
Let $P\in \fB_{p,2}(n, m)$. Then we have the following.
\begin{enumerate}
\item Suppose $P=JK$ is a central decomposition. Let $J'=J[P,P]$, and $K'=K[P,P]$. 
Then $J'$ and $K'$ form a central decomposition of $P$, and both of them properly 
contain $[P,P]$.
\item\label{item: 2} If for a central subgroup $N$, $P/N$ 
admits a central decomposition, then $P/(N\cap [P,P])$ admits a central 
decomposition.
\item For any $g\in P$, $\deg(g)\leq n-1$.
\end{enumerate}
\end{observation}
\begin{proof}
(1): To show that $J'$ and $K'$ form a central decomposition of $P$, we only need 
to verify that $J'$ and $K'$ are proper. For the sake of 
contradiction, 
suppose $P=J'=J[P,P]$. Since $[P,P]$ is the Frattini subgroup of $P$,
it follows 
that $J=P$, contradicting that $J$ is proper. 

To show that $J'$ properly contains $[P,P]$, again for the sake of contradiction 
suppose $J'\leq [P,P]$. Then $P=J'K'\leq [P,P]K'=K'$, a contradiction to $K'$ 
being a proper subgroup of $P$.
%
%
%

(2): If $N\leq [P,P]$, the conclusion holds trivially. Suppose otherwise. Let 
$J/N$ and $K/N$ be a 
central product of $P/N$ for $J, K\leq P$. That is, for any $j\in J$ and $k\in K$, 
$jkj^{-1}k^{-1}\in N$, so in fact $jkj^{-1}k^{-1}=[j,k]\in N\cap [P,P]$. It then 
follows easily that $J/(N\cap [P,P])$ and $K/(N\cap [P,P])$ form a central product 
of $P/(N\cap [P,P])$.

(3): If $g\in\ZC(P)$, $\deg(g)=0$. If $g\not\in\ZC(P)$, then $C_P(g)$
contains the 
subgroup generated by $g$ and $[P, P]$, which is of order $\geq p^{m+1}$.
\end{proof}

Recall that in Step~\ref{item: BLT step 3}, we start from
bilinear 
map $\phi:\F_p^n\times \F_p^n\to \F_p^m$ satisfying $\phi(\F_p^n, \F_p^n)=\F_p^m$, 
and construct $P_\phi$,
a $p$-group 
of class $2$ 
and exponent $p$. 
Then $[P_\phi, P_\phi]\cong 
\Z_p^m$, and $P_\phi/[P_\phi, P_\phi]\cong \Z_p^n$. 

It is easily checked that, by Equation~\ref{eq:gp_def}, 
subspaces of $\F_p^m$ correspond to subgroups 
of $[P_\phi, P_\phi]$, and subspaces of $\F_p^n$ correspond to subgroups of 
$P_\phi/[P_\phi, P_\phi]$. We then set up the following notation. 
For $U\leq \F_p^n$,
let $Q_U$ be the subgroup 
of $P_\phi/[P_\phi, P_\phi]$ corresponding to $U$, and 
let $S_U$ 
be the \emph{smallest} subgroup of $P_\phi$ satisfying $S_U[P,P]/[P,P]=Q_U$. Note 
that $S_U$ is regular with respect to commutation, that is, $S_U\cap [P,P]=[S_U, 
S_U]$.
For $X\leq 
\F_p^m$, 
let $N_X$ 
be the subgroup of 
$[P_\phi, P_\phi]$ corresponding to $X$. 


\begin{proposition}\label{prop:map_group}
Let $\phi:\F_p^n\times \F_p^n\to\F_p^m$ and $P_\phi\in\fB_{p,2}(n, m)$ be as 
above. 
Then $\kappa(\phi)=\kappa(P_\phi)$, and 
$\lambda(\phi)=\lambda(P_\phi)$. 
\end{proposition}
\begin{proof}
To show that $\kappa(\phi)\geq \kappa(P_\phi)$,
suppose there exists a 
$(n-\kappa(\phi))$-dimensional 
 $U\leq \F_p^n$ such that $\phi|_U$ admits an orthogonal 
decomposition. It can be verified easily that this induces a central decomposition 
for the regular subgroup $S_U\leq P_\phi$. Furthermore, by the second isomorphism 
theorem, $S_U/[S_U,S_U]=S_U/(S_U\cap[P_\phi, P_\phi])\cong 
S_U[P_\phi,P_\phi]/[P_\phi, P_\phi]=Q_U$, which is of order $p^{n-\kappa(\phi)}$.

To show that $\kappa(\phi)\leq \kappa(P_\phi)$,
suppose that a regular $S\leq 
P_\phi$ satisfying $|S/[S,S]|=p^{n-\kappa(P_\phi)}$ admits a central decomposition 
$S=JK$.
Appyling Observation~\ref{obs:def} (1) to $S$, we can assume that $J$ and $K$ both 
properly 
contain 
$[S,S]$. Let $U_S$ (resp. $U_J$, $U_K$) 
be the subspace of $\F_p^n$ corresponding 
to $S[P_\phi, P_\phi]/[P_\phi, P_\phi]$ (resp. $J[P_\phi, P_\phi]/[P_\phi, 
P_\phi]$, 
$K[P_\phi, P_\phi]/[P_\phi, 
P_\phi]$). 
Then it can be verified, using Equation~\ref{eq:gp_def}, that $U_J$ and $U_K$ form 
an 
orthogonal decomposition for $\phi|_{U_S}$. Furthermore, by the second isomorphism 
theorem, $S[P_\phi, P_\phi]/[P_\phi, 
P_\phi]\cong S/[S,S]$, which holds with $S$ replaced by 
$J$ or $K$ as well. In particular we have $\dim(U_S)=n-\kappa(P_\phi)$.
%

To show that $\lambda(\phi)\geq \lambda(P_\phi)$,
we translate a subspace $X\leq 
\F_p^m$ 
such that $\phi/_X$ admits an orthogonal decomposition, to a subgroup 
$N_X\leq 
[P_\phi, P_\phi]$. 
Then it can be verified easily that the orthogonal 
decomposition of $\phi/_X$ yields a central decomposition of $P_\phi/N_X$. 

To show that $\lambda(\phi)\leq \lambda(P_\phi)$,
suppose $N\leq P_\phi$ 
is a 
central subgroup of order $p^{\lambda(P_\phi)}$ such that $P_\phi/N$ admits a 
central decomposition. 
By Observation~\ref{obs:def} (2), we can assume that $N\leq [P_\phi, P_\phi]$. 
Let $X$ be the subspace of $\F_p^m$ corresponding to $N$.
Let $J/N, K/N\leq P_\phi/N$ be a central decomposition of $P_\phi/N$ for $J, K\leq 
P_\phi$. Applying
Observation~\ref{obs:def} (1) to $P_\phi/N$, we can assume that $J/N$ and $K/N$ 
both properly contain $[P_\phi/N, P_\phi/N]=[P_\phi, P_\phi]/N$. In particular, 
$J$ and $K$ properly contain $[P_\phi, P_\phi]$, 
so
$J/[P_\phi, P_\phi]$ (resp. 
$K/[P_\phi, P_\phi]$) corresponds to a non-trivial 
proper subspace $U_J\leq \F_p^n$ (resp. $U_K\leq \F_p^n$).
Then it can be verified that $U_J$ and $U_K$ span $\F_p^n$, 
and 
for any $u\in U_J$ 
and $u'\in U_K$, 
we have $\phi(u,u')\in X$. 
Therefore $U_J$ 
and $U_K$ form an orthogonal decomposition for $\phi/_{X}$.
\end{proof}

\begin{remark}[On the regular condition]\label{rem:regularity}
The reason for imposing the regular condition is to rule out the following central 
decompositions, which is not well-behaved regarding the correspondence between 
$\phi$ and $P_\phi$. Suppose that $S\leq P_\phi$ satisfies 
$[S,S]<[P_\phi,P_\phi]$. Then $S$ and $[P_\phi,P_\phi]$ form a central 
decomposition of $S[P_\phi,P_\phi]$. Translating back to $\phi$, $[S,S]<[P_\phi, 
P_\phi]$ just says that $\phi(U_S, U_S)$ is a proper subspace of $\F^m$, which is 
not related to whether $\phi|_{U_S}$ admits an orthogonal decomposition. 
\end{remark}

\subsection{Proof of Proposition~\ref{prop:relation}}

We shall work in the setting of alternating matrix spaces. So we state the 
correspondence of Definition~\ref{def:gp_degree} in 
this setting, which was proposed in \cite{Qia17} and has been used in 
\cite{BCG+19}.

\begin{definition}[Degrees and $\delta$ for alternating matrix spaces]
\label{def:alt_degree}
Let $\cA\leq \Lambda(n, \F)$. For $v\in \F^n$, the \emph{degree} of $v$ in $\cA$ 
is the dimension of $\cA v:=\{Av:A\in\cA\}$. The \emph{minimum degree} of $\cA$, 
denoted as $\delta(\cA)$, is the minimum degree over all $0\neq v\in\F^n$. 
\end{definition}

To translate from groups in $\fB_{p,2}(n, m)$ to alternating matrix spaces, we 
recall the following procedure which consists of inverses of the last two steps of 
the Baer-Lov\'asz-Tutte procedure. 

For any $P\in\fB_{p,2}(n, m)$, let $V=P/[P,P]\cong \Z_p^n$ and 
$U=[P,P]\cong\Z_p^m$.
The commutator map $\phi_P:V\times V\to U$ is alternating 
and bilinear. After fixing bases of $V$ and $U$ as $\F_p$-vector spaces, we can 
represent $\phi_P:\F_p^n\times\F_p^n\to\F_p^m$ as $(A_1, \dots, A_m)\in\Lambda(n, 
\F_p)^m$, 
which spans an $m$-dimensional 
$\cA_P\leq \Lambda(n, \F_p)$. 
It is easy to 
check that isomorphic groups yield isometric alternating matrix spaces. 
Furthermore, this procedure preserves $\kappa$ and $\lambda$, by essentially the 
same proof for 
Proposition~\ref{prop:space_map} and~\ref{prop:map_group}, and $\delta$, by a 
straightforward calculation. 
%

The following proposition then implies Proposition~\ref{prop:relation} (1).
\begin{proposition}
Given $\cA\leq\Lambda(n,\F)$, we have $\kappa(\cA)\leq\delta(\cA)$ and $\lambda(\cA)\leq\delta(\cA)$. 
\end{proposition}
\begin{proof}
We first show that $\kappa(\cA)\leq \delta(\cA)$. Take 
some $v\in \F^n$ such that $\deg(v)=\delta(\cA)$. If 
$\delta(\cA)=n-1$, then the inequality holds trivially. 
Otherwise, let 
$U=\{u\in\F^n: \forall\ A\in\cA, u^tAv=0\}$. Note that $\dim(U)=n-\deg(v)\geq 2$, 
and $v\in U$. Let $V$ be 
any complement space of $\langle v\rangle$ in $U$. Then $\langle v\rangle \oplus 
V$ is an orthogonal decomposition of $\cA|_U$. It follows that $\kappa(\cA)\leq 
n-\dim(U)=\deg(v)=\delta(\cA)$.
%

We then show that $\lambda(\cA)\leq \delta(\cA)$. Take 
some $v\in\F^n$ such that $\deg(v)=\delta(\cA)$. Let $W$ be any complement 
subspace 
of $\langle v\rangle$ in $\F^n$, and let $T_W$ be an $n\times (n-1)$ matrix whose 
columns form a basis of $W$. The space $v^t\cA T_W=\{v^tAT_W : A\in 
\cA\}\leq \M(1\times (n-1), \F)$ 
is of dimension $\deg(v)$. By Equation~\ref{eq:alt_def}, we then have 
$\lambda(\cA)\leq \dim(v^t\cA T_W)=\deg(v)=\delta(\cA)$. 
%
\end{proof}

In contrast to the graph setting, we show that it is possible that 
$\kappa(\cA)>\lambda(\cA)$ over $\Q$ and $\F_q$, therefore proving 
Proposition~\ref{prop:relation} (2). For this we need the following definition.
\begin{definition}\label{def:fully_connect}
We say that $\cA\leq \Lambda(n, \F)$ is \emph{fully connected}, if for any 
linearly 
independent $u, v\in \F^n$, there exists $A\in \cA$, such that $u^tAv\neq 0$. 
\end{definition}

An observation on fully connected $\cA$ follows from the definition easily.

\begin{observation}\label{obs:full}
Suppose that $\cA\leq \Lambda(n, \F)$ is fully connected. Then $\kappa(\cA)=n-1$.
\end{observation}

We shall construct a fully connected $\cA\leq \Lambda(n, \F)$ with 
$\lambda(\cA)<n-1=\kappa(\cA)$. To do this we need the fully connected notion in 
the (not 
necessarily alternating) matrix space setting. 
That is, $\cB\leq \M(s\times t, \F)$ is fully connected, if for any nonzero 
$u\in 
\F^s$ and nonzero $v\in \F^t$, 
there 
exists $B\in \cB$, such that $u^tBv\neq 0$. The following fact is well-known. 
\begin{fact}\label{fact:full}
Let $\F$ be a finite field or $\Q$. Then over $\F$, there exists a fully 
connected matrix space in $\M(s, \F)$ of dimension $s$.
\end{fact}
\begin{proof}
Let $\mathbb{K}$ be a degree-$s$ field extension of $\F$. The regular 
representation of $\mathbb{K}$ on $\F^s$ gives an $s$-dimensional $\cC\leq \M(s, 
\F)$, such that each nonzero $C\in \cC$ is of full rank. Let $(C_1, \dots, C_s)$ 
be an ordered basis of $\cB$. Let $B_i\in\M(s, \F)$, $i\in [s]$, be defined by 
$B_i=\begin{bmatrix}
C_1{e_i} & C_2{e_i} & \dots & C_s{e_i}
\end{bmatrix}$. That is, the $j$th column of $B_i$ is the $i$th column of $C_j$. 
Then $\cB=\langle B_1, \dots, B_s\rangle\leq\M(s, \F)$ is of dimension $s$ and 
fully connected. Indeed, if $\cB$ is not fully connected, then there exist 
nonzero $v\in \F^s$ and nonzero $u=\begin{bmatrix}
u_1 \\ u_2 \\ \vdots \\ u_s
\end{bmatrix}\in \F^s$ such that $v^tB_iu=0$ for any $i\in [s]$. But this just 
means that $v$ is in the left kernel of $C'=u_1C_1+\dots+u_sC_s$, contradicting 
that $C'$ is of full rank.
\end{proof}


Let $s, t\in \N$ and $n=s+t$. Let $\cB\leq \M(s\times t, \F)$ be a fully connected 
matrix space of dimension 
$d<n-1$. We shall use $\cB$ to 
construct a fully connected 
$\cA\leq \Lambda(n, \F)$ such that $\lambda(\cA)\leq d<n-1= \kappa(\cA)$. 

Suppose $\cB$ is spanned by $B_1, \dots, B_d\in \M(s\times t, \F)$. Let 
$A_i=\begin{bmatrix} 
0 & B_i \\
-B_i^t & 0 \end{bmatrix}$ for $i\in[d]$. For $1\leq i<j\leq s$, let 
$C_{i,j}=\begin{bmatrix}
E_{i,j} & 0 \\
0 & 0
\end{bmatrix}\in \Lambda(n, \F)$, where 
$E_{i,j}={e_i}e_j^t-{e_j}e_i^t\in\Lambda(s,\F)$ is an elementary alternating matrix. 
For $1\leq i<j\leq t$, let 
$D_{i,j}=\begin{bmatrix} 0 & 0 \\ 0 & F_{i,j}\end{bmatrix}\in\Lambda(n, \F)$, 
where $F_{i,j}={e_i}e_j^t-{e_j}e_i^t\in\Lambda(t,\F)$ is an 
elementary alternating 
matrix. Let $\cA$ be spanned by $\{A_i:i\in[d]\}\cup \{C_{i,j}:1\leq i<j\leq 
s\}\cup \{D_{i,j}:1\leq i<j\leq t\}$. 
\begin{proposition}
Let $\cA\leq \Lambda(n, \F)$ be as above. Then $\cA$ is fully connected. 
\end{proposition}
\begin{proof}
Assume there exist linearly independent $u,v\in\F^n$ such that for any $A\in\cA$, $u^tAv=0$. 
Take $u=\begin{bmatrix}u_1\\u_2\end{bmatrix}$ and $v=\begin{bmatrix}v_1\\v_2\end{bmatrix}$, where $u_1,v_1\in\F^s$ and $u_2,v_2\in\F^t$. Note that for any $1\leq i<j\leq s$, 
$$
\begin{bmatrix}
u_1^t&u_2^t
\end{bmatrix}
\begin{bmatrix}
E_{i,j} & 0 \\
0 & 0
\end{bmatrix}
\begin{bmatrix}
v_1\\v_2
\end{bmatrix}
=u_1^tE_{i,j}v_1=0.$$
Similarly, we have $u_2^tF_{i,j}v_2=0$ for all $1\leq i<j\leq t$. 

We then 
distinguish among the following cases.
\begin{enumerate}
\item $v_1$ and $v_2$ are both nonzero. In this case we have $u_1=\lambda v_1$ 
and $u_2=\mu v_2$ for some 
$\lambda\neq\mu\in\F$. Therefore, we have
$$\begin{bmatrix}
u_1^t&u_2^t
\end{bmatrix}
\begin{bmatrix}
0 & B_i \\
-B_i^t & 0
\end{bmatrix}
\begin{bmatrix}
v_1\\v_2
\end{bmatrix}=-u_2^tB_i^tv_1+u_1^tB_iv_2=-\mu v_2^tB_i^tv_1+\lambda 
v_1^tB_iv_2=(\lambda-\mu)v_1^tB_iv_2.$$
Since $\cB$ is fully connected, this implies that $v_1=0$ or $v_2=0$, a 
contradiction to the assumption of this case. 
\item $v_1$ is zero and $v_2$ is nonzero. Then $u_2=\lambda v_2$, and 
$u_1$ cannot be zero. Therefore, we have
$$\begin{bmatrix}
u_1^t&u_2^t
\end{bmatrix}
\begin{bmatrix}
0 & B_i \\
-B_i^t & 0
\end{bmatrix}
\begin{bmatrix}
v_1\\v_2
\end{bmatrix}=-u_2^tB_i^tv_1+u_1^tB_iv_2=u_1^tB_iv_2=0, $$
which is a contradiction to the full connectivity of $\cB$. 
\item $v_1$ is nonzero and $v_2$ is zero. This case is in complete analogy with 
the previous case. 
\end{enumerate}
This concludes the proof that $\cA$ is fully connected. 
%
\end{proof} 
We then have 
$\kappa(\cA)=n-1$ by Observation~\ref{obs:full}. Now observe that the subspace 
of $\cA$ spanned by $C_{i,j}$ and $D_{i,j}$ admits a 
central decomposition. This gives that $\lambda(\cA)\leq 
d<n-1=\kappa(\cA)$. Over $\F_q$ and $\Q$, such $\cB$ exists for 
$s>1$ by Fact~\ref{fact:full}. This concludes the proof of 
Proposition~\ref{prop:relation} (2). 

\bibliographystyle{alpha}
\bibliography{references}
\end{document}